\documentclass{amsart}
\usepackage{graphicx}
\usepackage{amssymb}
\usepackage{amsfonts}
\usepackage{bbm}
\swapnumbers
\newtheorem{theorem}{Theorem}[section]

\newtheorem{corollary}[theorem]{Corollary}

\theoremstyle{definition}

\numberwithin{equation}{section}
 \theoremstyle{plain}
 
 \numberwithin{equation}{section} 
 \numberwithin{figure}{section} 
 \theoremstyle{plain}
 \theoremstyle{remark}
 \newtheorem*{acknowledgement*}{Acknowledgement}

\newcommand{\cF}{{\mathcal F}}

\newcommand{\Om}{{\Omega}}

\newcommand{\ve}{{\varepsilon}}
\newcommand{\del}{{\delta}}

\newcommand{\gam}{{\gamma}}
\newcommand{\Gam}{{\Gamma}}

\newcommand{\Sig}{{\Sigma}}
\newcommand{\sig}{{\sigma}}
\newcommand{\al}{{\alpha}}

\newcommand{\ka}{{\kappa}}

\newcommand{\La}{{\Lambda}}

\newcommand{\bbR}{{\mathbb R}}

\begin{document}
\title[]{Iterated ergodic theorems and Erd\" os--R\' enyi Law of Large Numbers}%
 \vskip 0.1cm
 \author{ Yuri Kifer\\
\vskip 0.1cm
 Institute  of Mathematics\\
Hebrew University\\
Jerusalem, Israel}%
\address{
Institute of Mathematics, The Hebrew University, Jerusalem 91904, Israel}
\email{ kifer@math.huji.ac.il}%

\thanks{ }
\dedicatory{  }
 \date{\today}
\begin{abstract}\noindent
We obtain ergodic theorems and a version of the Erd\" os--Renyi law of large numbers 
for multiple iterated sums and integrals 
of the form $\Sigma^{(\nu)}(t)=\sum_{0\leq k_1<...<k_\nu\leq t}\xi(k_1)\otimes\cdots\otimes\xi(k_\nu)$,
$t\in[0,T]$ and $\Sigma^{(\nu)}(t)=\int_{0\leq s_1\leq...\leq s_\nu\leq t}\xi(s_1)\otimes\cdots\otimes\xi(s_\nu)ds_1\cdots ds_\nu$ where $\{\xi(k)\}_{-\infty<k<\infty}$ and $\{\xi(s)\}_{-\infty<s<\infty}$ are vector processes for which standard ergodic theorems, i.e. when $\nu=1$,  hold true. 

\end{abstract}
\maketitle
\markboth{Yu.Kifer}{Iterated ergodic theorems}
\renewcommand{\theequation}{\arabic{section}.\arabic{equation}}
\pagenumbering{arabic}

\section{Introduction}\label{sec1}\setcounter{equation}{0}

Let $\{\xi(k)\}_{-\infty<k<\infty}$ and $\{\xi(t)\}_{-\infty<t<\infty}$ be discrete and continuous time $d$-dimensional
 stochastic processes.
The collections or sums of multiple iterated sums
\begin{equation}\label{1.1}
\Sig^{(\nu)}(n)=\sum_{0\leq k_1<...<k_\nu<n}\xi(k_1)\otimes\cdots\otimes\xi(k_\nu),
\end{equation}
in the discrete time, and iterated integrals
\begin{equation}\label{1.2}
\Sig^{(\nu)}(t)=\int_{0\leq u_1\leq...\leq u_\nu\leq t}\xi(u_1)\otimes\cdots\otimes\xi(u_\nu)du_1\cdots du_\nu,
\end{equation}
in the continuous time, were called signatures in recent papers related to the rough paths theory, data science, machine learning and neural networks (see, for instance, \cite{BFT}, \cite{DET} and references there).

In this paper we will derive ergodic theorems for the above iterated sums and integrals, showing that
 as $t\to\infty$ the normalized expressions $S^{(\nu)}(t)=t^{-\nu}\Sig^{(\nu)}(t)$ for $\nu>1$ converge provided
  the corresponding limits exist when $\nu=1$. This enables us to derive at the end a version of the Erd\" os--Renyi 
  law of large numbers for iterated sums and integrals. Observe that in
 \cite{Ki23} and \cite{Ki25} we obtained strong moment and almost sure limit theorems which, of course, imply
 ergodic theorems type limits but this requires quite strong weak dependence in time conditions
 on stochastic processes while here we obtain iterated ergodic theorems under,
  essentially, the same assumptions as in the classical $\nu=1$ case.
 
 We observe that there is a series of papers on ergodic theorems for $U$-statistics (see, for instance, 
  \cite{G24} and references there) but the assumptions there do not cover our setup here and the methods 
 here and there are quite different.

\section{Preliminaries and main results}\label{sec2}\setcounter{equation}{0}

\subsection{Iterated sums}\label{subsec2.1}

Let $\xi(0),\,\xi(1),\,\xi(2),...$ be a sequence of $d$-dimensional vectors $\xi(k)=(\xi_1(k),...,\xi_d(k)),\, k=0,1,...$.
We will consider iterated sums of the form
\begin{equation}\label{2.1}
\Sig^{(\nu)}(n)=\sum_{0\leq k_1<...<k_\nu<n}\xi(k_1)\otimes\cdots\otimes\xi(k_\nu)
\end{equation}
where the left hand side can be viewed as a $d^\nu$--dimensional vector with coordinates
\begin{equation}\label{2.2}
\Sig^{i_1,...,i_\nu}(n)=\sum_{0\leq k_1<...<k_\nu<n}\xi_{i_1}(k_1)\xi_{i_2}(k_2)\cdots\xi_{i_\nu}(k_\nu),
\end{equation}
$1\leq i_1,...,i_\nu\leq d$. We will prove

\begin{theorem}\label{thm2.1}
Suppose that
\begin{equation}\label{2.3}
\limsup_{n\to\infty} n^{-1}\sum_{k=0}^n|\xi(k)|=R<\infty
\end{equation}
and the limit
\begin{equation}\label{2.4}
\lim_{n\to\infty}n^{-1}\sum_{k=0}^n\xi(k)=Q=(Q_1,...,Q_d)
\end{equation}
exists with $|Q|<\infty$. Then for any $1\leq i_1,...,i_\nu\leq d$,
\begin{equation}\label{2.5}
\lim_{n\to\infty}n^{-\nu}\Sig^{i_1,...,i_\nu}(n)=\frac 1{\nu !}\prod_{j=1}^\nu Q_{i_j}.
\end{equation}
\end{theorem}

The main application of the above theorem is to the case when $\xi(k),\, k\geq 0$ is a stationary vector 
stochastic process on a probability space $(\Om,\cF,P)$ with
\begin{equation}\label{2.6}
E|\xi(0)|<\infty.
\end{equation}
The process $\xi$ can be represented by a $P$-preserving dynamical system $F:\Om\to\Om$ and a measurable 
function $g$ on $\Om$ so that $\xi(k)=g\circ F^k$. If (\ref{2.6}) holds true then (\ref{2.3}) and (\ref{2.4})
 are satisfied with probability one by the Birkhoff ergodic theorem (see, for instance, \cite{Kal}
 or \cite{Wal}), and so Theorem \ref{thm2.1} yields (\ref{2.5}) $P$-almost surely in this case.
 Thus, we have
 \begin{corollary}\label{cor2.2}
 Let $\xi(k)=(\xi_1(k),...,\xi_d(k)),\, k\geq 0$ be a stationary $d$-dimensional stochastic process such that
 (\ref{2.6}) holds true. Let $Q$ be an almost sure random (in general) vector limit in (\ref{2.4}) which exists by 
 the Birkhoff ergodic theorem. Then (\ref{2.5}) holds true with probability one where $\Sig^{i_1,...,i_\nu}$ is the 
 iterated sum constructed by $\xi$ as in (\ref{2.2}).
 \end{corollary}
 
 Next, we formulate the mean ergodic theorem for iterated sums.
 \begin{theorem}\label{thm2.3}
 Let $\xi(k)=(\xi_1(k),...,\xi_d(k)),\, k=0,1,...$ be a stationary $d$-dimensional stochastic process such that
 \begin{equation}\label{2.7}
 E|\xi(0)|^\nu<\infty
 \end{equation}
 and
 \begin{equation}\label{2.8}
 E\xi(0)=Q=(Q_1,...,Q_d).
 \end{equation}
 Then
 \begin{equation}\label{2.9}
 \lim_{n\to\infty}E|n^{-\nu}\Sig^{i_1,...,i_\nu}(n)-\frac 1{\nu !}\prod_{j=1}^\nu Q_{i_j}|=0
 \end{equation}
 for any $1\leq i_1,...,i_\nu\leq d$.
 \end{theorem}
 
 We observe that if we assume in (\ref{2.7}) that $E|\xi(0)|^{\nu p}<\infty$ for some $p\geq 1$ then we can
 obtain the convergence in (\ref{2.8}) in the $L^p$-norm.

\subsection{Iterated integrals}\label{subsec2.2}

Here we start with a continuous in $t$ path $\xi(t),\, t\geq 0$ in $\bbR^d$ such that all integrals $\int_0^t\xi(s)ds$
and $\int_0^t|\xi(s)|ds,\, t\geq 0$ exist. We will consider iterated integrals of the form
\begin{equation}\label{2.10}
\Sig^{(\nu)}(t)=\int_{0\leq s_1\leq...\leq t}\xi(s_1)\otimes\cdots\otimes\xi(s_\nu)ds_1\cdots ds_\nu
\end{equation}
where the left hand side can be viewed as a $d^\nu$-dimensional vector with the coordinates
\begin{equation}\label{2.11}
\Sig^{i_1,...,i_\nu}(t)=\int_{0\leq s_1\leq...\leq t}\xi_{i_1}(s_1)\cdots\xi_{i_\nu}(s_\nu)ds_1\cdots ds_\nu
\end{equation}
where $1\leq i_1,...,i_\nu\leq d$. We will prove the continuous time version of Theorem \ref{thm2.1},
\begin{theorem}\label{thm2.4}
Suppose that
\begin{equation}\label{2.12}
\limsup_{n\to\infty} t^{-1}\int_{0}^t|\xi(s)|ds=R<\infty
\end{equation}
and the limit
\begin{equation}\label{2.13}
\lim_{t\to\infty}t^{-1}\int_{0}^t\xi(s)ds=Q=(Q_1,...,Q_d)
\end{equation}
exists and $|Q|<\infty$. Then for any $1\leq i_1,...,i_\nu\leq d$,
\begin{equation}\label{2.14}
\lim_{t\to\infty}t^{-\nu}\Sig^{i_1,...,i_\nu}(t)=\frac 1{\nu !}\prod_{j=1}^\nu Q_{i_j}.
\end{equation}
\end{theorem}

Again, the main application of the above theorem is to the case when $\xi(t),\, t\geq 0$ is a stationary vector
stochastic process on a probability space $(\Om,\cF,P)$ with (\ref{2.6}) being satisfied. In particular, the
process $\xi$ can be represented by a $P$-preserving dynamical system (flow) $F^t:\Om\to\Om,\, t\geq 0$ and a 
function $g$ on $\Om$ so that $\xi(t)=g\circ F^t$. If (\ref{2.6}) holds true then (\ref{2.12}) and (\ref{2.13}) are satisfied by the Birkhoff ergodic theorem (see, for instance, \cite{Kal}, Ch. 10), and so Theorem \ref{thm2.4} yields (\ref{2.14})
in this case with probability one.

Thus we have
\begin{corollary}\label{cor2.5} Let $\xi(t)=(\xi_1(t),...,\xi_d(t)),\, t\geq 0$ be a continuous in $t$ stationary
$d$-dimensional stochastic process such that (\ref{2.6}) holds true. Let $Q$ be an almost sure random (in general)
 vector limit in (\ref{2.13}) which exists by the Birkhoff ergodic theorem. Then (\ref{2.14}) holds true with
 probability one where $\Sig^{i_1,...,i_\nu}$ is the iterated integral constructed by $\xi$ as in (\ref{2.11}).
 \end{corollary}

The following convergence in mean assertion holds true in the continuous time, as well.
\begin{theorem}\label{thm2.6} 
Let $\xi(t)=(\xi_1(t),...,\xi_d(t)),\, t\geq 0$ be a continuous in $t$ stationary $d$-dimensional stochastic process
 such that (\ref{2.7}) and (\ref{2.8}) hold true.  Then
 \begin{equation}\label{2.15}
 \lim_{t\to\infty}E|t^{-\nu}\Sig^{i_1,...,i_\nu}(t)-\frac 1{\nu!}\prod^\nu_{j=1}Q_{i_j}|=0
 \end{equation}
 for any $1\leq i_1,...,i_\nu\leq d$.
 \end{theorem}

\section{Discrete time case proofs}\label{sec3}\setcounter{equation}{0}

\subsection{Abel's summation by parts}\label{subsec3.1}

In what follows we will use several times the classical Abel's summation by parts formula
(which is used also, for instance, in \cite{Or}),
\begin{equation}\label{3.1}
\sum_{r=0}^na_rb_r=\sum_{r=0}^{n-1}(r+1)(a_r-a_{r+1})\sig_r+(n+1)a_n\sig_n,
\end{equation}
where $\sig_r=(r+1)^{-1}\sum_{q=0}^rb_q$, which is valid for any sequences $a_0,a_1,...$ and $b_0,b_1,...$.

For reader's convenience we will recall its simple proof which goes on by induction. For $n=1$ we have 
\[
a_0b_0+a_1b_1=(a_0-a_1)b_0+2a_1\frac {(b_0+b_1)}2=(a_0-a_1)\sig_0+2a_1\sig_1.
\]
Suppose that (\ref{3.1}) holds true for $n=1,2,...,m$ and prove it for $n=m+1$. Using (\ref{3.1}) for
$n=m$ we have
\begin{eqnarray*}
&\sum_{r=0}^{m+1}a_rb_r=\sum_{r=0}^ma_rb_r+a_{m+1}b_{m+1}=\sum_{r=0}^{m-1}(r+1)(a_r-a_{r+1})\sig_r\\
&+(m+1)a_m\sig_m+a_{m+1}b_{m+1}=\sum_{r=0}^m(r+1)(a_r-a_{r+1})\sig_r+(m+2)a_{m+1}\sig_{m+1},
\end{eqnarray*}
as required.

\subsection{Proof of Theorem \ref{thm2.1}}\label{subsec3.2}

We proceed by induction. For $\nu=1$ the result follows from our assumption (\ref{2.4}). Suppose that
the assertion of Theorem \ref{thm2.1} holds true for $\nu=1,...,n-1,\, n\geq 2$ and we will prove its 
validity for $\nu=n$. We have
\[
\Sig^{i_1,...,i_n}(m)=\sum_{0\leq k<m}\xi_{i_n}(k)\Sig^{i_1,...,i_{n-1}}(k)
\]
and
\[
S^{i^{(n)}}(m)=S^{i_1,...,i_n}(m)=m^{-n}\sum_{0\leq k<m}\xi_{i_n}(k)k^{n-1}S^{i_1,...,i_{n-1}}(k)
\]
where $i^{(n)}=(i_1,...,i_n)$. We set also
\[
Q_{i^{(l)}}=Q_{i_1,...,i_l}=\prod_{j=1}^lQ_{i_j}\quad\mbox{and}\quad \del_{i^{(l)}}(k)=S^{i^{(l)}}(k)
-\frac 1{l!}Q_{i^{(l)}}.
\]

By the induction hypothesis 
\begin{equation}\label{3.2}
\lim_{m\to\infty}S^{i^{(n-1)}}(m)=\lim_{m\to\infty}S^{i_1,...,i_{n-1}}(m)=\frac {Q_{i^{(n-1)}}}{(n-1)!}.
\end{equation}
Next, we write
\[
S^{i^{(n)}}(m)=m^{-n}\sum_{0\leq k<m}\xi_{i_n}(k)k^{n-1}(\del_{i^{(n-1)}}(k)+\frac {Q_{i^{(n-1)}}}{(n-1)!})
=I_1^{i^{(n)}}(m)+I_2^{i^{(n)}}(m)
\]
where
\[
I_1^{i^{(n)}}(m)=\frac {Q_{i^{(n-1)}}}{(n-1)!}m^{-n}\sum_{0\leq k<m}k^{n-1}\xi_{i_n}
\]
and
\[
I_2^{i^{(n)}}(m)=m^{-n}\sum_{0\leq k<m}k^{n-1}\xi_{i_n}(k)\del_{i^{(n-1)}}(k).
\]

We start with $I_1^{i^{(n)}}$ and invoke (\ref{3.1}) to obtain that for $n>1$,
\begin{eqnarray*}
&I_1^{i^{(n)}}(m)=\frac {Q_{i^{(n-1)}}}{(n-1)!}m^{-n}(\sum_{r=0}^{m-2}(r+1)(r^{n-1}-(r+1)^{n-1})\sig_r^{i_n}\\
&+m(m-1)^{n-1}\sig_{m-1}^{i_n})=\frac {Q_{i^{(n-1)}}}{(n-1)!}m^{-n}\big(Q_{i_n}\sum_{r=0}^{m-2}(r+1)(r^{n-1}-(r+1)^{n-1})\\
&+\sum_{r=0}^{m-2}(r+1)(r^{n-1}-(r+1)^{n-1})\del_{i_n}(r)+Q_{i_n}m(m-1)^{n-1}+m(m-1)^{n-1}\del_{i_n}(m-1)\big)
\end{eqnarray*}
where $\sig_r^i=(r+1)^{-1}\sum_{q=0}^r\xi_i(q)=S^i(r)$ and $\del_i(r)=\sig^i_r-Q_i=S^i(r)-Q_i$.

Since $(r+1)^{n-1}=r^{n-1}+(n-1)r^{n-2}+...+1$, it follows that
\begin{equation}\label{3.3}
\lim_{n\to\infty}m^{-n}\sum_{r=0}^{m-2}(r+1)(r^{n-1}-(r+1)^{n-1})=-\frac {(n-1)}n.
\end{equation}
By the assumption (\ref{2.4}),
\begin{equation}\label{3.4}
\lim_{m\to\infty}\del_i(m)=0\quad\mbox{for any}\quad i=1,...,d,
\end{equation}
and so in order to derive that
\begin{equation}\label{3.5}
\lim_{m\to\infty}I_1^{i^{(n)}}(m)=\frac {Q_{i^{(n)}}}{n!},
\end{equation}
it suffices to show that
\begin{equation}\label{3.6}
\lim_{m\to\infty}m^{-n}\sum_{r=0}^{m-2}(r+1)(r^{n-1}-(r+1)^{n-1})\del_{i_\nu}(r)=0.
\end{equation}

Since
\[
-2n\leq m^{-(n-1)}(r+1)(r^{n-1}-(r+1)^{n-1})\leq 0
\]
for any $r\leq m-2$, in order to obtain (\ref{3.5}) it suffices to show that
\begin{equation}\label{3.7}
\lim_{m\to\infty}m^{-1}\sum_{r=0}^m|\del_i(r)|=0\quad\mbox{for}\quad i=1,...,d.
\end{equation}
To derive (\ref{3.7}) observe that by (\ref{3.4}) we can define for each $\ve>0$ and $j=1,...,d$,
\[
K_{j,\ve}=\min\{ l:\, |\del_j(m)|<\ve\,\,\,\mbox{for all}\,\,\, m>l\}<\infty.
\]
Then
\[
m^{-1}\sum_{r=0}^m|\del_j(r)|\leq\ve+m^{-1}\sum_{r=0}^{K_{j,\ve}}|\del_j(r)|.
\]
Hence, 
\[
\limsup_{m\to\infty}m^{-1}\sum_{r=0}^m|\del_j(r)|\leq\ve . 
\]
Since $\ve>0$ is arbitrary, we obtain that this upper limit is equal to zero implying (\ref{3.7}) which yields 
(\ref{3.6}), and so (\ref{3.5}) follows.

In order to complete the proof of Theorem \ref{thm2.1} it remains to show that
\begin{equation}\label{3.8}
\lim_{m\to\infty}I_2^{i^{(n)}}(m)=0.
\end{equation}
By (\ref{3.1}),
\begin{equation}\label{3.9}
I_2^{i^{(n)}}(m)=m^{-n}\sum_{r=0}^{m-2}(r+1)(r^{n-1}-(r+1)^{n-1})\eta_{i^{(n)}}(r)+\frac {m-1}{m^{n-1}}\eta_{i^{(n)}}(m-1)
\end{equation}
where
\[
\eta_{i^{(n)}}(l)=\frac 1{l+1}\sum_{r=0}^l\xi_{i_n}(r)\del_{i^{(n-1)}}(r).
\]
By the induction hypothesis
\[
\lim_{k\to\infty}\del_{i^{(n-1)}}(k)=0,
\]
and so, similarly to the above, we can define
\[
K_{i^{(n-1)},\ve}=\min\{ l:\,|\del_{i^{(n-1)}}(k)|<\ve\quad\mbox{for all}\quad k>l\}<\infty.
\]
Then
\[
|\eta_{i^{(n)}}(m)|\leq\frac \ve{m+1}\sum_{r=0}^m|\xi_{i_n}(r)|+\frac 1{m+1}\sum_{r=0}^{K_{i^{(n-1)},\ve}}|\xi_{i_n}(r)|
|\del_{i^{(n-1)}}(r)|,
\]
and so by the assumption (\ref{2.3}),
\[
\limsup_{m\to\infty}|\eta_{i^{(n)}}(m)|\leq R\ve.
\]
Hence, by (\ref{3.9}) for $n\geq 2$,
\[
\limsup_{m\to\infty}|I_2^{i^{(n)}}(m)|\leq R\ve
\]
and since $\ve>0$ is arbitrary, we obtain (\ref{3.8}) and complete the proof of Theorem \ref{thm2.1}. \qed

\subsection{Proof of Theorem \ref{thm2.3}}\label{subsec3.3}

We will give two proofs of Theorem \ref{thm2.3}: one direct and one based on Corollary \ref{cor2.2} and 
the approximation of the process $\xi$ by bounded processes. In the direct proof we will rely first on the $L^p$
ergodic theorem (see, for instance, \cite{Wal},\,\S 1.6) which yields from (\ref{2.7}) and (\ref{2.8}) that
\begin{equation}\label{3.10}
\lim_{m\to\infty}\| m^{-1}\sum_{k=0}^m\xi(k)-Q\|_\nu=0
\end{equation}
where $\|\cdot\|_p$ denotes the $L^p$-norm. We will prove by induction that for each $n=1,2,...,\nu$ and $i^{(n)}=(i_1,...,i_n)$,
\begin{equation}\label{3.11}
\lim_{m\to\infty}\|\del_{i^{(n)}}(m)\|_{\nu/n}=\lim_{m\to\infty}\| S^{i_1,...,i_n}(m)-\frac {Q_{i^{(n)}}}{n!}\|_{\nu/n}=0.
\end{equation}

Indeed, for $n=1$ we have (\ref{3.11}) by (\ref{3.10}). Suppose that (\ref{3.11}) holds true for $n=1,...,k-1<\nu$
and prove it for $n=k$. As in the previous subsection we will write
\[
S^{i^{(k)}}(m)=I_1^{i^{(k)}}(m)+I_2^{i^{(k)}}(m)
\]
with $I_1^{i^{(k)}}$ and $I_2^{i^{(k)}}$ defined there but now $S^{i^{(k)}}(m)$, $I_1^{i^{(k)}}(m)$ and
$I_2^{i^{(k)}}(m)$ are random variables. The induction step and the whole proof of Theorem \ref{thm2.3} will
be completed if we show that
\begin{equation}\label{3.12}
\lim_{m\to\infty}\| I_1^{i^{(k)}}(m)-\frac {Q_{i^{(k)}}}{k!}\|_{\nu/k}=0
\end{equation}
and
\begin{equation}\label{3.13}
\lim_{m\to\infty}\| I_2^{i^{(k)}}(m)\|_{\nu/k}=0
\end{equation}
taking into account that (\ref{3.12}) and (\ref{3.13}) follow from (\ref{3.10}) for $k=1$.

We saw in Section \ref{subsec3.2} that
\begin{eqnarray*}
& I_1^{i^{(k)}}(m)=\frac {Q_{i^{(k)}}}{(k-1)!}(m^{-k}\sum_{r=0}^{m-2}(r+1)(r^{k-1}-(r+1)^{k-1})\\
&+\frac {(m-1)^{k-1}}{m^{k-1}})+\frac{Q_{i^{(k-1)}}}{(k-1)!}(m^{-k}\sum_{r=0}^{m-2}(r+1)(r^{k-1}\\
&-(r+1)^{k-1})\del_{i_k}(r)+\frac {(m-1)^{k-1}}{m^{k-1}}\del_{i_k}(m-1)).
\end{eqnarray*}
Letting here $m\to\infty$ and taking into account (\ref{3.10}) we derive easily (\ref{3.12}).

Considering the representation (\ref{3.8}) of $I_2^{i^{(k)}}$ we conclude that
\begin{equation}\label{3.14}
\| I_2^{i^{(k)}}(m)\|_{\nu/k}\leq m^{-k}(k+1)\sum^{m-2}_{r=0}(r+1)^{k-2}\|\eta_{i^{(k)}}(r)\|_{\nu/k}+
\frac {m-1}{m^{k-1}}\|\eta_{i^{(k)}}(m-1)\|_{\nu/k}.
\end{equation}
and by the H\" older inequality
\[
\|\eta_{i^{(k)}}(l)\|_{\nu/k}\leq\frac 1{l+1}\sum_{r=0}^l\|\xi_{i_k}(r)\|_\nu\|\del_{i^{(k-1)}}(r)\|_{\frac \nu{k-1}}.
\]
By the induction hypothesis, (\ref{3.11}) holds true when $n=k-1$, and so
\[
\|\del_{i^{(k-1)}}(l)\|_{\frac \nu{k-1}}\to 0\,\,\mbox{as}\,\, l\to\infty.
\]
This together with (\ref{2.7}) yields that
\begin{equation}\label{3.15}
\lim_{l\to\infty}\|\eta_{i^{(k)}}(l)\|_{\nu/k}=0.
\end{equation}
Now we obtain (\ref{3.13}) from (\ref{3.14}) and (\ref{3.15}) similarly to the final part of the proof of
Theorem \ref{thm2.1}, completing the proof of Theorem \ref{thm2.3}.

Another proof of Theorem \ref{thm2.3} proceeds, essentially similar to the standard approximation argument in the $L^p$
ergodic theorem (see \cite{Wal}). Namely, set
\[
\zeta^{(N)}(k)=\xi(k)\mathbbm{1}_{|\xi(k)|\leq N}\quad\mbox{for}\,\,\, N=1,2,...\,\mbox{and}\,\,\, k=0,1,...
\]
where $\mathbbm{1}_\Gam$ is the indicator of a set $\Gam$. Then, 
\begin{equation}\label{3.16}
\|\zeta^{(N)}(k)-\xi(k)\|_{\nu}=\|\xi(k)\mathbbm{1}_{|\xi(k)|>N}\|_{\nu}=\rho_\nu(N)\,\downarrow 0\,\,\,\mbox{as}\,\,\,
N\uparrow\infty
\end{equation}
and $\rho_\nu(N)$ does not depend on $k$. Now let $\Sig_{N}^{i_1,...,i_\nu}$ be the iterated sum constructed as 
$\Sig^{i_1,...,i_\nu}$ but with $\zeta_{i_1}^{(N)},...,\zeta_{i_\nu}^{(N)}$ in place of $\xi_{i_1},...,\xi_{i_\nu}$.
Then by Corollary \ref{cor2.2} with probability one,
\begin{equation}\label{3.17}
lim_{m\to\infty} m^{-\nu}\Sig_N^{i_1,...,i_\nu}(m)=\frac 1{\nu !}\prod_{j=1}^\nu Q_{i_j}^{(N)}
\end{equation}
where almost surely,
\[
Q_i^{(N)}=\lim_{m\to\infty} m^{-1}\sum_{k=0}^m\zeta_i^{(N)}(k)
\]
taking into account that the last limit exists with probability one by the Birkhoff ergodic theorem.
Observe that
\[
|m^{-\nu}\Sig_N^{i_1,...,i_\nu}(m)|\leq N^\nu,
\]
and so (\ref{3.17}) together with the dominated convergence theorem yields
\begin{equation}\label{3.18}
\lim_{m\to\infty}E|m^{-\nu}\Sig_N^{i_1,...,i_\nu}(m)-\frac 1{\nu !}\prod_{j=1}^\nu Q_{i_j}^{(N)}|=0.
\end{equation}

Next, observe that by the H\" older inequality,
\begin{eqnarray*}
&E|\xi_{i_1}(k_1)\xi_{i_2}(k_2)\cdots\xi_{i_\nu}(k_\nu)-\zeta_{i_1}^{(N)}(k_1)\zeta_{i_2}^{(N)}(k_2)\cdots
\zeta_{i_\nu}^{(N)}(k_\nu)|\\
&\leq\sum_{j=1}^\nu E|\xi_{i_1}(k_1)\cdots\xi_{i_{j-1}}(k_{j-1})(\xi_{i_j}(k_j)
-\zeta_{i_j}^{(N)}(k_j))\zeta_{i_{j+1}}^{(N)}(k_{i_{j+1}})\cdots\zeta_{i_\nu}^{(N)}(k_\nu)|\\
&\leq\|\xi_{i_1}\|_\nu\|\xi_{i_2}\|_{\nu}\cdots\|\xi_{i_{j-1}}\|_{\nu}\rho_\nu(N)\|\zeta^{(N)}_{i_{j+1}}\|_{\nu}
\cdots\|\zeta_{i_\nu}^{(N)}\|_{\nu}\leq\nu\|\xi\|_{\nu}^{\nu-1}\rho_\nu(N).
\end{eqnarray*}
Combining this with (\ref{3.16})--(\ref{3.17}) we derive easily (\ref{2.9}). \qed

\section{Continuous time case}\label{sec4}
\subsection{Proof of Theorem \ref{thm2.4}}\label{subsec4.1}
It is possible to derive Theorem \ref{thm2.4} from Theorem \ref{thm2.1} comparing iterated sums of
$\zeta(n)=\int_{n-1}^n\xi(s)ds$ with iterated integrals of the process $\xi$ itself but when $\nu>1$ this comparison is not so easy as in the standard situation $\nu=1$. We prefer to provide a direct proof seeing, in particular, that the Abel
summation formula we used in the discrete time case corresponds to the integration by parts when dealing with iterated
 integrals.

For $n>1$ we have now,
\[
\Sig^{i_1,...,i_n}(t)=\int_0^t\xi_{i_n}(s)\Sig^{i_1,...,i_{n-1}}(s)ds
\]
and
\begin{equation}\label{3.19}
S^{i^{(n)}}(t)=S^{i_1,...,i_n}(t)=t^{-n}\int_0^t\xi_{i_n}(s)s^{n-1}S^{i_1,...,i_{n-1}}(s)ds
\end{equation}
where both $\Sig$ and $S$ are considered to be 0 at 0.

We proceed by induction. For $\nu=1$ Theorem \ref{thm2.4} follows from the assumption. Suppose that (\ref{2.14}) holds true
for all $\nu\leq n-1$ and prove it for $\nu=n$. By the induction hypothesis
\begin{equation}\label{3.20}
\lim_{t\to\infty}S^{i^{(n-1)}}(t)=\lim_{t\to\infty}S^{i_1,...,i_{n-1}}(t)=\frac {Q_{i^{(n-1)}}}{(n-1)!}.
\end{equation}
By (\ref{3.19}) we can write
\[
S^{i^{(n)}}(t)=t^{-n}\int_0^t\xi_{i_n}(s)s^{n-1}(\del_{i^{(n-1)}}(s)+\frac {Q_{i^{(n-1)}}}{(n-1)!})ds=
I_1^{i^{(n)}}(t)+I_2^{i^{(n)}}(t)
\]
where, as before, 
\begin{eqnarray*}
&\del_{i^{(l)}}(s)=S^{i^{(l)}}(s)-\frac {Q_{i^{(l)}}}{l!},\,\,
I_1^{i^{(n)}}(t)=\frac {Q_{i^{(n-1)}}}{(n-1)!}t^{-n}\int_0^ts^{n-1}\xi_{i_n}(s)ds,\\
&I_2^{i^{(n)}}(t)=t^{-n}\int_0^ts^{n-1}\xi_{i_n}(s)\del_{i^{(n-1)}}(s)ds\,\,\mbox{and}\,\,\del_j(s)=\frac 1s\int_0^s
\xi_j(u)du-Q_j.
\end{eqnarray*}

Next, we will integrate by parts the integral in $I_1^{i^{(n)}}$ to obtain
\begin{eqnarray*}
&I_1^{i{(n)}}(t)=\frac {Q_{i^{(n-1)}}}{(n-1)!}t^{-n}\big(-(n-1)\int_0^ts^{n-2}ds\int_0^s\xi_{i_n}(u)du+t^{n-1}\int_0^t
\xi_{i_n}(s)ds\big)\\
&=\frac {Q_{i^{(n-1)}}}{(n-1)!}t^{-n}\big(-Q_{i_n}t^n\frac {(n-1)}n+(n-1)\int_0^ts^{n-1}\del_{i_n}(s)ds+t^nQ_{i_n}+
t^n\del_{i_n}(t)\big)\\
&=\frac {Q_{i^{(n)}}}{n!}-\frac {Q_{i^{(n-1)}}}{(n-2)!}t^{-n}\int_0^ts^{n-1}\del_{i_n}(s)ds+
\frac {Q_{i^{(n-1)}}}{(n-1)!}\del_{i_n}(t).
\end{eqnarray*}
The last term here tends to zero as $t\to\infty$ since $\del_{i_n}(t)\to 0$ by the assumption (\ref{2.13}) of 
Theorem \ref{thm2.4}. We have also
\[
t^{-n}|\int_0^ts^{n-1}\del_{i_n}(s)ds|\leq\frac 1t\int_0^t|\del_{i_n(s)}|ds\to 0\,\,\,\mbox{as}\,\,\, t\to\infty
\]
where we argue in the same way as in (\ref{3.7}). Hence,
\begin{equation}\label{3.21}
\lim_{t\to\infty}I_1^{i^{(n)}}(t)=\frac {Q_{i^{(n)}}}{n!}.
\end{equation}

In view of (\ref{3.21}), in order to complete the induction step and the whole proof of Theorem \ref{thm2.4} it remains
to show that
\begin{equation}\label{3.22}
\lim_{t\to\infty}I_2^{i^{(n)}}(t)=0.
\end{equation}
We invoke the integration by parts again to obtain
\begin{equation}\label{3.23}
I_2^{i^{(n)}}(t)=t^{-n}\big(-(n-1)\int_0^ts^{n-1}\eta_{i^{(n)}}(s)ds+t^{n}\eta_{i^{(n)}}(t)\big)
\end{equation}
where $\eta_{i^{(n)}}(v)=v^{-1}\int_0^v\xi_{i_n}(u)\del_{i^{(n-1)}}(u)du$. 

The final steps of the proofs of Theorems \ref{thm2.1} and \ref{thm2.4} are similar. Namely, by (\ref{3.20}) we have here
\[
\lim_{t\to\infty}\del_{i^{(n-1)}}(t)=0,
\]
and so we can define for each $\ve>0$ also in the continuous time case
\[
K_{i^{(n-1)},\ve}=\inf\{ t:\, |\del_{i^{(n-1)}}(s)|<\ve\,\,\mbox{for all}\,\, s>t\}<\infty.
\]
Then
\[
|\eta_{i^{(n)}}(t)|\leq t^{-1}\ve\int_0^t|\xi_{i_n}(u)|du+t^{-1}\int_0^{K_{i^{(n-1)},\ve}+1}|\xi_{i_n}(u)||\del_{i^{n-1}}(u)|du,
\]
and so by the assumption (\ref{2.12}),
\[
\limsup_{t\to\infty}|\eta_{i^{(n)}}(t)|\leq R\ve.
\]
Hence, by (\ref{3.23}) for $n\geq 2$,
\[
\limsup_{t\to\infty}|I_2^{(i^{(n)}}(t)|\leq R\ve
\]
and since $\ve>0$ is arbitrary, we obtain (\ref{2.14}) completing the proof of Theorem \ref{thm2.4}.  \qed

\subsection{Proof of Theorem \ref{thm2.6}}\label{subsec4.2}

The proof of Theorem \ref{thm2.6} goes on in the same way as the proof of Theorem \ref{thm2.3}. Especially easy
to adapt to the continuous time case the second proof of Theorem \ref{thm2.3} appearing in Section 3.3.

Set again $\zeta^{(N)}(t)=\xi(t)\mathbbm{1}_{|\xi(t)|\leq N}\quad\mbox{for}\,\,\,N=1,2,...$ and $t\geq 0$.
Then for any $n=0,1,...,\nu$ and $t\geq 0$,
\begin{equation}\label{3.24}
\|\zeta^{(N)}(t)-\xi(t)\|_{\nu}=\|\xi(t)\mathbbm{1}_{|\xi(t)|\geq N}\|_{\nu}=\|\xi(0)\mathbbm{1}_{|\xi(t)|\geq N}\|_{\nu}
=\rho_\nu(N)\,\downarrow 0\,\,\,\mbox{as}\,\,\,N\uparrow\infty
\end{equation}
where $\rho_\nu(N)$ does not depend on $t$. Now let $\Sig_N^{i_1,...,i_\nu}$ be the iterated integral given by (\ref{2.11})
but with $\zeta_i^{(N)}$'s in place of $\xi_i$'s. Then by Corollary \ref{cor2.5} with probability one,
\begin{equation}\label{3.25}
\lim_{t\to\infty}t^{-\nu}\Sig_N^{i_1,...,i_\nu}(t)=\frac 1{\nu!}\prod_{j=1}^\nu Q^{(N)}_{i_j}
\end{equation}
where almost surely
\[
Q_i^{(N)}=\lim_{t\to\infty}t^{-1}\int_0^t\zeta_i^{(N)}(s)ds
\]
and the latter holds true by the continuous time Birkhoff ergodic theorem (see, for instance, \cite{Kal}).

Observe that
\[
|t^{-\nu}\Sig_N^{i_1,...,i_\nu}(t)|\leq N^\nu,
\]
and so (\ref{3.25}) together with the dominated convergence theorem yields
\begin{equation}\label{3.26}
\lim_{t\to\infty}E|t^{-\nu}\Sig_N^{i_1,...,i_\nu}(t)-\frac 1{\nu!}\prod_{j=1}^\nu Q^{(N)}_{i_j}|=0.
\end{equation}
Using the H\" older inequality we obtain now
\begin{eqnarray*}
&E|\Sig^{i_1,...,i_\nu}(t)-\Sig_N^{i_1,...,i_\nu}(t)|\leq\int_{0\leq s_1\leq...\leq s_\nu\leq t}
E|\xi_{i_1}(s_1)\cdots\xi_{i_\nu}(s_\nu)\\
&-\zeta^{(N)}_{i_1}(s_1)\cdots\zeta^{(N)}_{i_\nu}(s_\nu)|ds_1\cdots ds_\nu\\
&\leq\int_{0\leq s_1\leq...\leq s_\nu\leq t}\sum_{j=1}^\nu E|\xi_{i_1}(s_1)\cdots\xi_{i_{j-1}}(s_{j-1})
(\xi_{i_j}(s_j)\\
&-\zeta^{(N)}_{i_j}(s_j))\zeta^{(N)}_{i_{j+1}}(s_{j+1})\cdots\zeta^{(N)}_{i_\nu}(s_\nu)|ds_1\cdots ds_\nu\\
&\leq t^\nu\sum_{j=1}^\nu\|\xi_{i_1}\|_\nu\|\xi_{i_2}\|_{\nu}\cdots\|\xi_{i_{j-1}}\|_{\nu}\rho_\nu(N)\|\zeta^{(N)}_{i_{j+1}}
\|_{\nu}\cdots\|\zeta_{j_\nu}^{(N)}\|_{\nu}\\
&\leq t^\nu\nu\|\xi\|_{\nu}^{\nu-1}\rho_\nu(N).
\end{eqnarray*}
Combining this with (\ref{3.24})--(\ref{3.26}) we derive (\ref{2.15}) completing the proof of Theorem \ref{thm2.6}.

\section{Erd\" os--R\' enyi type law of large numbers for iterated sums and integrals}\label{sec5}\setcounter{equation}{0}
As an application of Theorem \ref{thm2.1} we will derive in this section a version of the Erd\" os--R\' enyi law of
large numbers (see \cite{ER}) for iterated sums and integrals which requires additional assumptions. 
A part of our proof will be a slight modification of the one from \cite{DK} and we will pick our additional assumptions 
from there.

As before, $\{\xi(k)\}_{-\infty<k<\infty}$ and $\{\xi(t)\}_{-\infty<t<\infty}$ will be discrete and continuous 
time $d$-dimensional stationary stochastic processes but now we assume also that they are bounded. Let $\cF_n^m$ and
$\cF_s^t$ be $\sig$-algebras generated by $\{\xi(k),\, n\leq k\leq m\|$ and by $\{\xi(u),\, s\leq u\leq t\}$, respectively.
We will proceed with the discrete time case while the continuous time case is dealt with in the same way.
As in \cite{DK} we assume that the above processes are $\psi$-mixing which means that
\begin{equation}\label{5.0}
\psi(n)=\sup_{k,\, A\in\cF_{-\infty}^k,\, B\in\cF_{k+n}^\infty,\, P(A)P(B)\ne 0}\big\vert\frac 
{P(A\bigcap B)}{P(A)P(B)}-1\big\vert\to 0\,\,\,\mbox{as}\,\,\, n\to\infty.
\end{equation}
Morover, following \cite{DK} we assume that the $\psi$-mixing coefficient $\psi(n)$ decays in $n$ exponentially fast.
It follows from
Theorem 3.3 in \cite{DK} and Section 6.4 in \cite{DZ} that under the above conditions the processes $\xi$ satisfy the large deviations principle with a rate function $I$. This implies also that each coordinate process $\xi_i, i=1,...,d$ satisfies the large deviations principle with a rate function $I_i$.
\begin{theorem}\label{thm5.1} Under the above conditions almost surely,
\begin{equation}\label{5.2}
\lim_{n\to\infty}\max_{0\leq m\leq n-\ell_n}I_{i_\nu}(\al)\frac {\Sig^{i_1,...,i_\nu}
(m+\ell_n)-\Sig^{i_1,...,i_\nu}(m)}{n^{\nu-1}\log n}=\frac {\al}{(\nu-1)!}\prod_{j=1}^{\nu-1}Q_{i_j}
\end{equation}
for any $1\leq i_1,...,i_\nu\leq d$ and $\al\in (0,c^+_{i_\nu})$ where $\ell_n=[(\log n)/I_{i_\nu}(\al)]$ and
$c^+_{i_\nu}=\lim_{n\to\infty}ess\sup n^{-1}\sum_{k=1}^n\xi_{i_\nu}(k)$.
\end{theorem}
\begin{proof}
 Employing Theorem \ref{thm2.1} we have almost surely
\begin{eqnarray}\label{5.3}
&\lim_{n\to\infty}\max_{0\leq m\leq n-\ell_n}I_{i_\nu}(\al)\frac {\Sig^{i_1,...,i_\nu}(m+\ell_n)-\Sig^{i_1,...,i_\nu}(m)}
{n^{\nu-1}\log n}\\
&=\lim_{n\to\infty}\max_{0\leq m\leq n-\ell_n}I_{i_\nu}(\al)\frac{\sum_{k=m+1}^{m+\ell_n}\xi_{i_\nu}(k)
\Sig^{i_\nu,...,i_{\nu-1}}(k-1)}{n^{\nu-1}\log n}\nonumber\\
&=\frac 1{(\nu-1)!}\prod_{j=1}^{\nu-1}Q_{i_j}\lim_{n\to\infty}\max_{0\leq m\leq n-\ell_n}\frac {I_{i_\nu}(\al)}
{n^{\nu-1}\log n}\sum_{k=m+1}^{m+\ell_n}\xi_{i_\nu}(k)\nonumber\\
&=\frac 1{(\nu-1)!}\prod_{j=1}^{\nu-1}Q_{i_j}\lim_{n\to\infty}\max_{n^\del\leq m\leq n-\ell_n}\frac {I_{i_\nu}(\al)m^{\nu-1}}
{n^{\nu-1}\log n}\sum_{k=m+1}^{m+\ell_n}\xi_{i_\nu}(k)\nonumber
\end{eqnarray}
for any $\del\in(0,1)$.

We will show that almost surely
\begin{equation}\label{5.4}
\lim_{n\to\infty}\max_{n^\del\leq m\leq n-\ell_n}\frac {I_{i_\nu}(\al)m^{\nu-1}}
{n^{\nu-1}\log n}\sum_{k=m+1}^{m+\ell_n}\xi_{i_\nu}(k)=\al
\end{equation}
which together with (\ref{5.3}) will yield (\ref{5.2}). We proceed similarly to \cite{DK} with a few modifications.
Let $\ve>0$ and define the event
\[
A_n(\ve)=\{\max_{0\leq m\leq n-l_n}m^{\nu-1}\sum_{k=m+1}^{m+\ell_n}\xi_{i_\nu}(k)\geq\ell_n(\al-\ve)n^{\nu-1}\}.
\]
Clearly,
\[
A_n(\ve)\subset\tilde A_n(\ve)=\{\max_{0\leq m\leq n-\ell_n}\sum_{k=m+1}^{m+\ell_n}\xi_{i_\nu}(k)\geq\ell_n(\al-\ve)\}.
\]
Relying on the Borel--Cantelli lemma it was shown in \cite{DK} that the event $\tilde A(\ve)$, and so the event 
$A(\ve)$, may occur only finitely many times in $n$. Since $\ve>0$ is arbitrary we conclude from here that almost surely,
\begin{equation}\label{5.5}
\limsup_{n\to\infty}\max_{0\leq m\leq n-\ell_n}\frac {I_{i_\nu}(\al)m^{\nu-1}}{n^{\nu-1}\log n}\sum_{k=m+1}^{m+\ell_n}\xi_{i_\nu}(k)\leq\al.
\end{equation}

Next, we will obtain the lower bound. Set
\[
B_n(\ve)=\{\max_{0\leq m\leq n-\ell_n}m^{\nu-1}\sum_{k=m+1}^{m+\ell_n}\xi_{i_\nu}(k)\leq\ell_n(\al-\ve)n^{\nu-1}\}
\]
and
\[
C_{n,m}(\ve)=\{ m^{\nu-1}\sum_{k=m+1}^{m+\ell_n}\xi_{i_\nu}(k)\leq\ell_n(\al-\ve)n^{\nu-1}\}.
\]
Then for $n-\ell_n>q$, where $q>0$ will be large enough,
\[
B_n(\ve)=\bigcap_{m=0}^{n-\ell_n}C_{n,m}(\ve)\subset\bigcap_{s=0}^{r_n}C_{n,s(\ell_n+q)}
\]
where we set $r_n=[\frac {n-\ell_n}{\ell_n+q}]$. Since $C_{n,m}(\ve)\in\cF_{m+1}^{m+\ell_n}$ we have by (\ref{5.0}) that
\begin{equation}\label{5.6}
P(B_n(\ve))\leq P\big(\bigcap_{s=0}^{r_n}C_{n,s(\ell_n+q}\big)\leq (1+\psi(q))^{r_n}\prod_{s=0}^{r_n}P(C_{n,s(\ell_n+q)}).
\end{equation}

By the assamption there exists $\La>0$ such that for all $q>0$ large enough
\begin{equation}\label{5.7}
\psi(q)\leq\exp(-\La q)
\end{equation}
and we choose $q=q_n=[\gam\log n]$ where $\gam\La>1$. It follows from our assumptions (see the beginning of the proof
of Theorem 4.1 in \cite{DK}) that $I_{i_\nu}(u)$ is strictly increasing in $u>0$, and so there exists $\del>0$ such that
\begin{equation}\label{5.8}
(I_{i_\nu}(\al-\frac \ve 2)+\del)/I_{\i_\nu}(\al)<1-\del.
\end{equation}
Let $C^c_{s(\ell_n+q)}$ be the complement of $C_{s(\ell_n+q)}$, then using the stationarity of the process $\xi$ we have
\begin{eqnarray}\label{5.9}
&\quad 1-P(C_{s(\ell_n+q})=P(C^c_{s(\ell_n+q})=P\{\sum_{k=s(\ell_n+q)+1}^{s(\ell_n+q)+\ell_n}\xi(k)>(s(\ell_n+q))^{-(\nu-1)}\\
&\times\ell_n(\al-\ve)n^{\nu-1}\}=P\{\ell_n^{-1}\sum_{k=1}^{\ell_n}\xi(k)>(s(\ell_n+q))^{-(\nu-1)}(\al-\ve)n^{\nu-1}\}.\nonumber
\end{eqnarray}

Next, chose $s$ so that $(s(\ell_n+q))^{-(\nu-1)}n^{\nu-1}\leq 1+\ka$ where $\ka>0$ is such that $(1+\ka)(\al-\ve)\leq\al-\frac \ve 2$, i.e. $s=s(n)\geq \frac n{\ell_n+q}(1+\ka)^{-\frac 1{\nu-1}}$. For such $s$ we apply the large deviations principle to $C^c_{s(\ell_n+q)}$ to obtain from (\ref{5.8}) and (\ref{5.9}) that for large $n$,
\begin{eqnarray}\label{5.10}
&1-P(C_{s(\ell_n+q})\geq P\{\ell_n^{-1}\sum_{k=1}^{\ell_n}\xi(k)>(\al-\ve)(1+\ka)\}\geq \\
&P\{\ell_n^{-1}\sum_{k=1}^{\ell_n}\xi(k)>(\al-\frac \ve 2)\}\geq\nonumber\\
&\exp(-\ell_n(I(\al-\frac \ve 2)+\del))\geq\exp(-\ell_nI(\al)(1-\del))\geq\exp(-(1-\del)\log n).\nonumber
\end{eqnarray}
Hence, by (\ref{5.7}) and (\ref{5.10}) for large $n$,
\begin{eqnarray}\label{5.11}
&P(B_n(\ve))\leq (1+\psi(q))^{r_n+1}\prod_{\frac n{\ell_n+q}(1+\ka)^{-\frac 1{\nu-1}}\leq s\leq r_n}P(C_{s(\ell_n+q)})\\
&\leq(1+e^{-\La\gam\log n})^{r_n+1}(1-\exp(-(1-\del)\log n))^{(1-(1+\ka)^{-\frac 1{\nu-1}})(r_n-1)}\nonumber\\
&\leq(1+n^{-1})^{r_n+1}(1-n^{-1+\del})^{(1-(1+\ka)^{-\frac 1{\nu-1}})(r_n-1)}\nonumber
\end{eqnarray}
where in the last expression and below we take $r_n=[\frac {n-\ell_n}{\ell_n+\gam\log n}]$. Observe that for large $n$,
\begin{equation}\label{5.12}
(1+n^{-1})^{r_n+1}\leq (1+n^{-1})^{n+1}<3.
\end{equation}
Set $\eta=1-(1+\ka)^{-\frac 1{\nu-1}}>0$. Then for large $n$,
\begin{equation}\label{5.13}
(1-n^{-1+\del})^{\eta (r_n-1)}=(1+\frac 1{(n-n^\del)n^{-\del}})^{-(n-n^\del)n^{-\del}(\frac {n^\del\eta (r_n-1)}{n-n^\del})}
\leq Ce^{-n^{\del/2}}
\end{equation}
for some constant $C>0$, since $\exp(-n^{\del/2}\frac {\eta (r_n-1)}{n-n^\del})\to 0$ as $n\to\infty$.

Now we obtain from (\ref{5.11})--(\ref{5.13}) that
\[
\sum_{n=1}^\infty P(B_n(\ve))<\infty
\]
and by the Borel--Cantelly lemma the event $B_n(\ve)$ can occur only finitely many times in $n$, and so
\begin{equation}\label{5.14}
\liminf_{n\to\infty}\max_{0\leq m\leq n-\ell_n}
\frac {m^{\nu-1}\sum_{k=m+1}^{m+\ell_n}\xi(k)}{\ell_nu^{\nu^{\nu-1}}}\geq\al-\ve.
\end{equation}
Since $\ve>0$ is arbitrary we obtain (\ref{5.2}) of Theorem \ref{thm5.1} from (\ref{5.3}), (\ref{5.5}) and (\ref{5.14}).
\end{proof}

\end{document}